\documentclass[a4paper,11pt]{amsart}

\usepackage{amsfonts}
\usepackage{amssymb}
\usepackage{graphics}

\newcommand{\dl}{\lambda}

\newcommand{\C}{\mathbb{C}}
\newcommand{\N}{\mathbb{N}}
\newcommand{\Z}{\mathbb{Z}}
\newcommand{\R}{\mathbb{R}}

\newcommand{\fol}{\mathcal{F}}

\newcommand{\calO}{{\mathcal{O}}}

\newcommand{\Diff}{{{\rm Diff}\, ({\mathbb C^n}, 0)}}

\newcommand{\diffn}{{{\rm Diff}\, ({\mathbb C^n}, 0)}}

\newcommand{\diffCtwo}{{{\rm Diff}_1 ({\mathbb C}^2, 0)}}

\newcommand{\Diffgentwo}{{{\rm Diff}\, ({\mathbb C}^2, 0)}}

\def\picill#1by#2(#3)#4
{\vbox to #2
{\hrule width #1 height 0pt depth 0pt
\vfill\special{illustration #3 scaled #4}}}

\newtheorem{theorem}{Theorem}
\newtheorem{prop}[theorem]{Proposition}

\newtheorem{lemma}[theorem]{Lemma}
\newtheorem{obs}[theorem]{Remark}

\theoremstyle{definition}
\newtheorem{defi}[theorem]{Definition}

\theoremstyle{remark}

\begin{document}

\title[Integrability and finite orbits]{A note on integrability and finite orbits for subgroups of $\Diff$}

\author{Julio C. Rebelo \, \, \, \& \, \, \, Helena Reis}
\address{}
\thanks{}

\maketitle

\begin{abstract}
In this note we extend to arbitrary dimensions a couple of results due respectively to Mattei-Moussu
and to Camara-Scardua in dimension~$2$.
We also provide examples of singular foliations having a Siegel-type singularity and answering in the negative
the central question left open in the previous work of Camara-Scardua.
\end{abstract}

\noindent \hspace{1.2cm} {\small AMS-Classification (2010): 37F75, \, 32H50, \, 37F99}

\bigskip

\section{Introduction}

This note concerns certain recently investigated aspects of higher dimensional generalizations of Mattei-Moussu's
celebrated topological characterization of integrable holomorphic foliations in dimension~$2$,
cf. \cite{M-M}. If $\fol$ denotes a singular holomorphic foliation defined about the origin of $\C^2$, then the
fundamental issue singled out in \cite{M-M} is the fact that the existence of holomorphic first integrals
for $\fol$ can be read off its topological dynamics. In particular, the existence of these first integrals can
be detected at the level of the topological dynamics of the holonomy pseudogroup of $\fol$. An immediate consequence
of their criterion is that the existence of (non-constant) holomorphic
first integrals for singular foliations as above is a property invariant by topological conjugation.

One basic question concerning higher dimensional generalizations of Mattei-Moussu's result
was to know whether a (local) singular holomorphic foliation on $(\C^n,0)$ topologically conjugate to another
holomorphic foliation possessing $n-1$ independent holomorphic first integrals should possess $n-1$ independent holomorphic
first integrals as well. This type of questions recently began to be investigated in \cite{BM}, \cite{scardua} partly
due to recent progresses in the study of the local dynamics of {\it parabolic}\, diffeomorphisms of $(\C^2,0)$,
see \cite{abate}, \cite{abate2}, \cite{raissy}.
In fact, a high point in Mattei-Moussu's argument is their proof that a pseudogroup of holomorphic
diffeomorphisms of $(\C,0)$ having {\it finite orbits}\, must itself be finite and the main theorem in \cite{BM}
is an extension of this to $(\C^2,0)$ under certain additional conditions. Similarly, in \cite{scardua}, the authors
considered the same question for foliations with Siegel singular points on $(\C^3,0)$ which is justified by the fact that these
Siegel foliations are the fundamental building blocks for the general case. By resorting to a celebrated result
due to M. Abate, Camara and Scardua affirmatively answered the question provided that the associated holonomy
maps have isolated fixed points.

Very recently, a general counterexample was found in \cite{thesis}. More precisely in \cite{thesis}
S.Pinheiro and the second author exhibited two topologically conjugate foliations on $(\C^3, 0)$ such that
one admits two independent holomorphic first integrals but not the other.
It so became clear that extending Mattei-Moussu theorem to higher dimensions is a subtle problem and it prompted
to a deeper analysis of the basic ingredients in the two-dimensional argument namely, the nature of pseudogroups
of $(\C^n,0)$ having finite orbits and the corresponding consequences for foliations having Siegel singular points.
The purpose of this Note is to contribute to the understanding of these questions by proposing an elementary generalization
of Mattei-Moussu result for pseudogroups of $(\C^n,0)$, which also yields a higher dimensional version of
the main theorem in \cite{scardua}, and by answering in the {\it negative}\, the general question
for Siegel singular points left open in \cite{scardua}.

To state our main results, let us work in the context of pseudogroups of $\diffn$
(the reader may check Section~2 for definitions and terminology). Our first result is a simple elaboration
of the corresponding statement in \cite{M-M} that turns out to generalize the corresponding result in
\cite{scardua} since it dispenses with the use of the deep
theorem on the existence of parabolic domains due to M. Abate \cite{abate} and valid only in dimension~$2$.

\vspace{0.2cm}

\noindent {\bf Theorem~A}. {\sl Let $G \subset \diffn$ be a finitely generated pseudogroup on a small neighborhood of the
origin in $\C^n$. Given $g \in G$, let ${\rm Dom}\, (g)$ denote the domain of definition of $g$ as element of the
pseudogroup in question. Suppose that for every $g \in G$ and $p \in {\rm Dom}\, (g)$ satisfying $g(p) =p$, one of
the following holds: either $p$ is an isolated fixed point of $g$ or $g$ coincides with the identity on a neighborhood of
$p$. Then the pseudogroup $G$ has finite orbits on a neighborhood of the origin if and only if $G$ itself is finite.}

\vspace{0.2cm}

\noindent {\bf Remark}. {\rm When $G$ is a subgroup of ${\rm Diff}\, (\C,0)$ the assumption of Theorem~A is automatically
verified so that the statement is reduced to Mattei-Moussu's corresponding result in \cite{M-M}. On the other hand,
it is proved in \cite{M-M} that a subgroup of ${\rm Diff}\, (\C,0)$ is not only finite but also cyclic.
In full generality, the second part of the statement cannot be generalized to higher dimensions since every finite
group embeds into a matrix group of sufficiently high dimension. In Section~2
the reader will find simple examples showing that, in fact, the group need not be cyclic already in
dimension~$2$ and even if the assumption of Theorem~A about ``isolated fixed points'' is satisfied.}

\vspace{0.2cm}

A two-dimensional variant of Theorem~A is proved in \cite{scardua} by resorting to Abate's theorem in \cite{abate}
(cf. Section~3 for a detailed comparison between the statement in \cite{scardua} and Theorem~A).
The authors of \cite{scardua} then go ahead to apply their result to the problem of
``complete integrability'' of differential equations. A similar application holds
in arbitrary dimensions and it will be discussed in Section~3. For the time being, it suffices to consider the situation
discussed in \cite{scardua} namely, $\fol$ is a foliation on $(\C^3,0)$ having a Siegel singular point
at the origin and leaving invariant the three ``coordinate axes''. Roughly speaking, the question addressed by
Camara-Scardua is to decide whether or not $\fol$ admits two independent holomorphic first integrals
(i.e. $\fol$ is ``completely integrable'') provided that the holonomy map associated to a certain invariant axis has finite orbits
(cf. Section~3 for accurate statements). Concerning the formulation of their result in this direction, is is however convenient
to mention an issue already pointed
out by Y. Genzmer in  his review to the article in question. In fact, whereas the methods of \cite{scardua} clearly
require the corresponding holonomy map to have isolated fixed points, the authors have failed to explicitly mention
this condition in the formulation of their main result. This said, the main problem left open by the work
of \cite{scardua} concerns precisely the validity of their result when
no assumption involving isolated fixed points is put forward. In other words, the question is whether or not
a Siegel singularity giving rise to a holonomy map with  finite orbits must be ``completely integrable''.
Though an affirmative answer to the latter question was expected, as pointed out by Genzmer and by Abate in
their reviews to the mentioned article, Theorem~B below shows that this is not the case.

\vspace{0.2cm}

\noindent {\bf Theorem~B}. {\sl Let $\fol$ denote the foliation associated to the vector field
\[
X = x(1 + x^2yz^3) \frac{\partial }{\partial x} + y(1 - x^2yz^3) \frac{\partial }{\partial y} - z \frac{\partial }{\partial z} \, .
\]
The foliation $\fol$ does not possess two independent holomorphic first integrals (though it possesses one
non-constant holomorphic first integral). Besides the holonomy map associated to the axis $\{ x=y=0\}$ has
finite orbits whereas it does not generate a finite subgroup of $\Diffgentwo$.}

\vspace{0.2cm}

In the above example, the reader will note that the restriction of $X$ to the invariant plane $\{ z=0\}$ yields
the radial vector field $x \partial /\partial x + y \partial /\partial y$ which admits the meromorphic first integral
$y/x$. This also answers in the negative a refinement of the initial question for Siegel singular point that had been
speculated by Mattei, namely whether the existence of a meromorphic first integral on the transverse plane plus the assumption
of finite orbits for the holonomy map might force the holonomy map in question to have finite order.

In Section~3, we shall also state and prove Theorem~\ref{whatever} which is deduced from our Theorem~A
and extends the main theorem in \cite{scardua} to arbitrarily high dimensions.

This short paper is organized as follows. Section~2 contains the proof of
Theorem~A along with the relevant definitions. As mentioned, Theorem~A is a simple elaboration of the arguments
in \cite{M-M}. Section~3 contains a small digression on Siegel singular points which enables us to state and
prove Theorem~\ref{whatever} extending to higher dimensions the result in \cite{scardua}. The proof of
Theorem~\ref{whatever}, in turn, amounts to a simple combination of Theorem~A and some useful results due to
P. Elizarov-Il'yashenko and to Reis, \cite{EI}, \cite{helena} connecting the linearization problem of these
singular foliations to the same question for certain holonomy maps. Finally, in Section~4 a
few interesting examples of local dynamics of diffeomorphisms tangent to the identity, along with local foliations realizing
some of them as local holonomy maps, will be provided. By building in these examples, the proof of Theorem~B will
quickly be derived.

\vspace{0.1cm}

\noindent {\bf Acknowledgments}. Most of this work was conducted during a visit of the first author to
IMPA and he would like to thank the CNPq-Brazil for partial financial support.
The second author was partially supported by FCT through CMUP.

\section{Proof of Theorem~A}

In the sequel, $G$ denotes a finitely generated subgroup of $\diffn$, where $\diffn$ stands for the group of germs of
local holomorphic diffeomorphisms of $\C^n$ fixing the origin. Assume that $G$
is generated by the elements $h_1, \ldots, h_k \in \diffn$. A natural way to make sense of the local dynamics of $G$ consists
of choosing representatives for $h_1, \ldots, h_k$ as local diffeomorphisms fixing $0 \in \C$. These representatives are still denoted
by $h_1, \ldots, h_k$ and, once this choice is made, $G$ itself can be identified to the {\it pseudogroup}\,
generated by these local diffeomorphisms on a (sufficiently small) neighborhood of the origin. It is then convenient to begin by
briefly recalling the notion of {\it pseudogroup}. For this, consider a small neighborhood $V$ of the origin where the
local diffeomorphisms $h_1, \ldots, h_k$, along with their inverses $h_1^{-1}, \ldots, h_k^{-1}$, are defined and one-to-one.
The pseudogroup generated by $h_1, \ldots, h_k$ (or rather by $h_1, \ldots , h_k, h_1^{-1}, \ldots, h_k^{-1}$ if there is any risk of confusion)
on $V$ is defined as follows. Every element of $G$ has the form $F = F_s \circ \ldots \circ F_1$ where each $F_i$,
$i \in \{1, \ldots, s\}$, belongs to the set $\{h_i^{\pm 1}, i=1, \ldots, k\}$. The element $F \in G$ should be regarded as an one-to-one holomorphic map
defined on a subset of $V$. Indeed, the domain of definition of $F = F_s \circ \ldots \circ F_1$, as an
element of the pseudogroup, consists of those points $x \in V$ such that for every $1 \leq l < s$ the point $F_l \circ
\ldots \circ F_1(x)$ belongs to $V$. Since the origin is fixed by the diffeomorphisms $h_1, \ldots, h_k$, it follows that
every element $F$ in this pseudogroup possesses a non-empty open domain of definition. This
domain of definition may however be disconnected. Whenever no misunderstanding is possible, the pseudogroup defined above will also
be denoted by $G$ and we are allowed to shift back and forward from $G$ viewed as pseudogroup or as group of germs.

Let us continue with some definitions that will be useful throughout the text. Suppose we are given local holomorphic diffeomorphisms $h_1,
\ldots, h_k, h_1^{-1}, \ldots, h_k^{-1}$ fixing the origin of $\C^n$. Let $V$ be a neighborhood of the origin where all these local diffeomorphisms
are defined and one-to-one. From now on, let $G$ be viewed as the pseudogroup acting on $V$ generated by these local diffeomorphisms.
Given an element $h \in G$, the domain of definition of $h$ (as element of $G$) will be denoted by ${\rm Dom}_V (h)$.

\begin{defi}
The $V_G$-orbit $\calO_V^G (p)$ of a point $p \in V$ is the set of points in $V$ obtained from $p$ by
taking its image through every element of $G$ whose domain of definition (as element of $G$) contains $p$.
In other words,
\[
\calO_V^G (p) = \{q \in V \; \, ; \; \,  q = h(p), \; h \in G \; \; {\rm and} \; \; p \in {\rm Dom}_V (h) \} \, .
\]
Fixed $h \in G$, the $V_h$-orbit of $p$ can be defined as the $V_{\langle h\rangle }$-orbit of $p$, where
$\langle h \rangle$ denotes the subgroup of $\diffn$ generated by $h$.
\end{defi}

We can now define ``pseudogroups with finite orbits". Note that
neighborhoods of the origin in $\C^n$ are always sufficiently small to ensure that
$h_1, \ldots , h_k, h_1^{-1}, \ldots, h_k^{-1}$ are well-defined injective maps on $V$.

\begin{defi}\label{def_finiteorbits}
A pseudogroup $G \subseteq \diffn$ is said to have finite orbits if there exists a sufficiently small open neighborhood $V$ of $0
\in \C^n$ such that the set $\calO_V^G (p)$
is finite for every $p\in V$. Analogously, $h \in G$ is said to have finite orbits if the pseudogroup $\langle h \rangle$
generated by $h$ has finite orbits.
\end{defi}

Fixed $h \in G$, the {\it number of iterations of $p$ by $h$}\, is the cardinality of the set $\{ n \in \Z \; \, ; \; \, p \in {\rm Dom}_V (h^{n}) \}$,
where ${\rm Dom}_V (h^{n})$ stands for the domain of definition of $h^n$ as element of the pseudogroup in question.
The number of iterations of $p$ by $h$ is denoted by $\mu_V^h (p)$ and belongs to $\N \cup \{\infty\}$. The lemma below is attributed
to Lewowicz and can be found in \cite{M-M}.

\begin{lemma}[Lewowicz]\label{lemmalewowicz}
Let $K$ be a compact connected neighborhood of $0 \in \R^n$ and $h$ a homeomorphism from $K$ onto $h(K) \subseteq \R^n$
verifying $h(0) = 0$. Then there exists a point $p$ on the boundary $\partial K$ of $K$ whose number of iterations
in $K$ by $h$ is infinite, i.e. $p$ satisfies $\mu_K^h (p) = \infty$.\qed
\end{lemma}

Fixed an open set $V$,  note that the existence of points in $V$ such that $\mu_K^h (p) = \infty$ does not imply that
$p$ is a point with infinite orbit, i.e. there may exist points $p$ in $V$ such that $\mu_V^h(p)=\infty$ but $\#
\calO_V^{\langle h \rangle }(p)<\infty$, where $\#$ stands for the cardinality of the set in question.
These points are called {\it periodic for $h$ on $V$}. A local diffeomorphism is said to be {\it periodic}\, if there is $k \in \N^{\ast}$
such that $f^k$ coincides with the identity on a neighborhood of the origin. Clearly periodic diffeomorphisms possess finite orbits.
To prove Theorem~A, we first need to show the following.

\begin{prop}\label{propperiodic}
Suppose that $G \subseteq \diffn$ is a group satisfying the condition of isolated fixed points of Theorem~A.
Let $h$ be an element of $G$ and assume that $h$ has only finite orbits.
Then $h$ is periodic.
\end{prop}

Assuming that Proposition~\ref{propperiodic} holds, Theorem~A can be derived as follows:

\begin{proof}[Proof of Theorem~A]
We want to prove that $G$ is finite (for example at the level of germs). So, let us consider the homomorphism
$\sigma : G \rightarrow GL(n, \C)$ assigning to an element $h \in G$ its derivative $D_0 h$ at the origin.
The image $\sigma (G)$ of $G$ is a finitely generated subgroup of $GL(n, \C)$ all of whose elements have finite order.
According to Schur's theorem concerning the affirmative solution of Burnside problem for linear groups, the group
$\sigma (G)$ must be finite, cf. \cite{burnside}. Therefore, to conclude that
$G$ is itself finite, it suffices to check that $\sigma$ is one-to-one or, equivalently, that its
kernel is reduced to the identity. Hence suppose that $h \in G$ lies in the kernel of $\sigma$, i.e. $D_0 h$
coincides with the identity. To show that $h$ itself coincides with the identity, note that $h$ must be
periodic since it has finite orbits, cf. Proposition~\ref{propperiodic}.
Therefore $h$ is conjugate to its linear part at the origin, i.e. it is conjugate to the identity map. Thus $h$ coincides
with the identity on a neighborhood of the origin of $\C^n$. Theorem~A is proved.
\end{proof}

Before proving Proposition~\ref{propperiodic}, let us make some comments concerning the proof of Theorem~A.
When $n=1$, Leau theorem immediately implies that the above considered homomorphism $\sigma$
is one-to-one so that $G$ will be abelian and, indeed, cyclic. This fact does not carry over higher dimensions since,
as already mentioned, every finite group
can be realized as a matrix group and therefore as a pseudogroup of $\Diff$ having finite orbits.
Yet, in general, groups obtained in this manner
contain non-trivial elements with non-isolated fixed points. Therefore, if we are dealing with
pseudogroups satisfying the conditions of
Theorem~A, the question on whether $G$ is abelian may still be raised. However,
even in this restricted setting the group $G$ need not be
abelian. For example, let $G$ be the subgroup of ${{\rm Diff}\, ({\mathbb C^2}, 0)}$
generated by $h_1(x,y) = (e^{\pi i/3} x, e^{2\pi i/3} y)$ and $h_2(x,y) = (y,x)$. Every element of $G$ has
finite orbits and possesses a single fixed point at the origin but the group $G$ is not abelian.

Let us now prove Proposition~\ref{propperiodic}. As already pointed out, the proof amounts to a careful reading of the
argument supplied in \cite{M-M} for the case $n=1$.

\begin{proof}[Proof of Proposition~\ref{propperiodic}]
Let $h$ be a local diffeomorphims in $\diffn$ whose periodic points are isolated unless the corresponding power of $h$
coincides with the identity on a neighborhood of the mentioned periodic point. Let us assume that $h$ is not periodic. To prove the
statement, we are going to show the existence of an open neighborhood $U$ of $0 \in \C^n$ such that the set of points
$x \in U$ with infinite $U_{\langle h \rangle}$-orbit is uncountable and has the origin as an accumulation point.
It will then result that $h$ cannot have finite orbits, thus proving the proposition.

Let $U$ be an arbitrarily small open neighborhood of $0 \in \C^n$ contained in the domains of definition of $h, \, h^{-1}$. Suppose also
that $h, \, h^{-1}$ are one-to-one on $U$. Consider
$\rho_0 > 0$ such that $D_{\rho_0} \subseteq U$, where $D_{\rho_0}$ stands for the closed ball of radius $\rho_0$ centered at the origin.
Following \cite{M-M}, we define the following sets
\begin{eqnarray*}
{\bf P} & = & \{x \in D_{\rho_0} : \; \mu_{D_{\rho_0}}(x) = \infty , \; \# \calO_{D_{\rho_0}}^{\langle h \rangle }(x) < \infty\} \, , \\
{\bf F} & = & \{x \in D_{\rho_0} : \; \mu_{D_{\rho_0}}(x) < \infty , \; \# \calO_{D_{\rho_0}}^{\langle h \rangle}(x) < \infty\} \, , \\
{\bf I} & = & \{x \in D_{\rho_0} : \; \mu_{D_{\rho_0}}(x) = \infty , \; \# \calO_{D_{\rho_0}}^{\langle h \rangle}(x) = \infty\} \, . \\
\end{eqnarray*}
In other words, ${\bf P}$ is the set of periodic points in $D_{\rho_0}$ for $h$, ${\bf F}$ denotes the set of points leaving $D_{\rho_0}$ after finitely many
iterations and ${\bf I}$ stands for the set of non-periodic points with infinite orbit. Naturally, $D_{\rho_0} = {\bf P} \cup {\bf F}
\cup {\bf I}$ and Lewowicz's lemma implies that
\[
({\bf P} \cup {\bf I}) \cap \partial D_\rho\neq \emptyset \, .
\]
for every $\rho \leq \rho_0$. Thus, at least one between ${\it P}$ and ${\bf I}$ is uncountable.
In what follows, the diffeomorphism $h$ is supposed to be non-periodic. With this assumption, our purpose is to show that
${\bf I}$ must be uncountable and the origin is accumulation point of ${\bf I}$.

For $n \geq 0$, let $A_n$ denote the domain of definition of $h^n$ viewed as an element of the pseudogroup {\it generated on $D_{\rho_0}$}. Clearly
$A_{n+1} \subseteq A_n$. Next, let $C_n$ be the connected (compact) component of $A_n$ containing the origin and pose
\[
C = \bigcap_{n \in \N} C_n \, .
\]
Note that $C$ is the intersection of a decreasing sequence of compact connected sets. Therefore $C$ is non-empty and connected.

\vspace{0.1cm}

\noindent {\it Claim}: Without loss of generality, the set $C$ can be supposed countable.

\begin{proof}[Proof of the Claim]
Suppose that $C$ is uncountable.
The reader is reminded that our aim is to conclude that ${\bf I}$ is uncountable provided that $h$ is not periodic. Therefore
we suppose for a contradiction that ${\bf I}$ is countable.
Since ${\bf I}$ is countable so is ${\bf I} \cap C$.
Consider now $C \cap {\bf P}$ and note that this intersection must be uncountable, since $C \subset {\bf P} \cup {\bf I}$. Let
\[
C \cap {\bf P} = \bigcup_{n\in\N} P_n \, ,
\]
where $P_n$ is the set of points $x \in C \cap {\bf P}$ of period $n$. Note that there exists a certain $n_0 \in \N$ such that
$P_{n_0}$ is infinite, otherwise all of the $P_n$ would be finite and $C \cap {\bf P}$ would be countable. Being infinite,
$P_{n_0}$ has a non-trivial accumulation point $p$ in $C_{n_0}$. The map $h^{n_0}$ is holomorphic on an open neighborhood
of $C_{n_0}$ and it is the identity on $P_{n_0} \cap C_{n_0}$. Since $p$ is not an isolated fixed point of $h^{n_0}$, it
follows that $h^{n_0}$ coincides with the identity map on $C_{n_0}$, i.e. on the connected component of the domain of
definition of $h^{n_0}$ that contains the origin. This contradicts the assumption of non-periodicity of $h$ (modulo reducing the
neighborhood of the origin). Hence ${\bf I}$
is uncountable. Moreover the closure of ${\bf I}$ contains the origin since, otherwise, there is a small disc
$D$ about the origin such that $D \cap {\bf I} =\emptyset$. If this is the case, it suffices
to repeat the above procedure with $C \cap D$ to obtain a contradiction.
\end{proof}

In view of the preceding, in the sequel $C$ will be supposed to consist of countably many points. The purpose is still to conclude
that the set ${\bf I}$ is uncountable (and the origin belongs to its closure).
Since, in addition, $C$ is connected, it must be reduced to the origin itself. Then,
for every $\rho < \rho_0$, we have $C \cap \partial D_\rho = \emptyset$. Now note that, for a fixed $\rho > 0$, the sets
\[
C_1 \cap \partial D_\rho, \, \,  (C_1 \cap C_2) \cap \partial D_\rho, \, \,  (C_1 \cap C_2 \cap C_3)\cap \partial D_\rho, \, \,  \ldots
\]
form a decreasing sequence of compact sets. Hence the intersection $\bigcap_{n\in\N} C_n \cap \partial D_\rho$ is non-empty,
unless there exists $n_0 \in \N$, such that $C_{n_0} \cap \partial D_\rho = \emptyset$. The latter case must occur since
$C \cap \partial D_\rho = \emptyset$. However, the value of $n_0$
for which the mentioned intersection becomes empty may depend on $\rho$.

Fix $\rho > 0$ and let $n_0$ be as above. Let $K$ be a compact connected neighborhood of $C_{n_0}$ that does not intersect
the other connected components of $A_{n_0}$, if they exist. The set $K$ can be chosen so that $\partial K \cap A_{n_0} = \emptyset$.
The inclusion $A_{n+1} \subseteq A_n$ guarantees that $\partial K$ does not intersect $A_n$, for every $n \geq n_0$.
Therefore
\[
\partial K \cap {\bf P} = \emptyset \, .
\]
In fact, if there were a periodic point $x$ of $D_{\rho_0}$ on $\partial K$, then $x$ would belong to every set $A_n$. In
particular, it would belong to $A_{n_0}$, hence leading to a contradiction. Nonetheless, Lewowicz's lemma guarantees the existence
of a point $x$ on the boundary $\partial K$ of $K$ such that the number of iterations in $K$ is infinite, i.e. such that
$\mu_K (x) = \infty$. Since $K \subseteq D_{\rho} \subseteq D_{\rho_0}$, it follows that
\[
\partial K \cap {\bf I} \ne \emptyset \, .
\]
By construction, it is clear that a compact set $K$ satisfying the above conditions is not unique. Indeed, for $K$ as before, denote
by $K_{\varepsilon}$ the compact connected neighborhood of $K$ whose boundary has distance to
$\partial K$ equal to $\varepsilon$. Then, there exists $\varepsilon_0 > 0$ such that $K_{\varepsilon}$ satisfies the
same properties as $K$ for every $0 \leq \varepsilon \leq \varepsilon_0$ with respect to $A_{n_0}$. In particular,
\[
\partial K_{\varepsilon} \cap {\bf I} \ne \emptyset
\]
for all $0 \leq \varepsilon \leq \varepsilon_0$. Therefore ${\bf I}$ must be uncountable. Finally, it remains to prove
that $0 \in \C^n$ is an accumulation point of ${\bf I}$. This is, however, a simple consequence of the fact that a compact set $K
\subseteq D_{\rho}$ as above can be considered for all $\rho > 0$. This completes the proof of Proposition~\ref{propperiodic}.
\end{proof}

\section{Siegel singular points and an extension of a result by Camara-Scardua}

We can now move on to state and prove Theorem~\ref{whatever}. The proof of this theorem follows from
the combination of our Theorem~A with the results in \cite{EI} or in \cite{helena}.

To begin with, let $\fol$ be a singular foliation on $(\C^n, 0)$ and let $X$ be a representative
of $\fol$, i.e. $X$ is a holomorphic vector field tangent to $\fol$ and whose singular set has codimension at
least~$2$. Suppose that the origin is a singular point of $\fol$ and denote by $\dl_1, \ldots, \dl_n$ the corresponding eigenvalues of $DX$ at the origin.
Assume the following holds:
\begin{enumerate}
\item $\fol$ has an isolated singularity the origin.
\item The singularity of $\fol$ is of Siegel type.
\item The eigenvalues $\dl_1, \ldots, \dl_n$ are all different from zero and
there exists a straight line through the origin, in the complex plane, separating $\dl_1$ from the remainder eigenvalues.
\item Up to a change of coordinates, $X = \sum_{i=1}^n \dl_ix_i(1+f_i(x)) \partial /\partial x_i$, where $x=(x_1,\ldots,x_n)$ and $f_i(0)=0$ for all $i$
\end{enumerate}
Note that condition~(4) is equivalent to the existence of $n$~invariant hyperplanes
through the origin. This condition, as well as condition~(3), is always verified when $n=3$ provided that the singular point is of
{\it strict Siegel type}\, cf. \cite{C}. Here the reader is reminded that, in dimension~$3$, the singular point is said to be
of strict Siegel type if $0 \in \C$ is contained in the interior of the convex hull of $\{\dl_1,\dl_2 , \dl_3\}$.

In general, the {\it eigenvalues of a foliation $\fol$}\, are nothing but the eigenvalues of the differential of a representative
vector field $X$ for $\fol$. These eigenvalues are therefore defined only up to a multiplicative constant and this definition
allows us to avoid passing through the choice of a representative vector field when dealing with foliations belonging to the
Siegel domain.

In any event, when it comes to problems of ``complete integrability'', the relevant Siegel singular points are not
of strict Siegel type. Indeed, in the case of strict Siegel singular point, the corresponding holonomy maps
have a ``hyperbolic part'' i.e. they are partially hyperbolic and therefore cannot have finite orbits. Similarly the
corresponding foliation cannot be ``completely integrable''. In these cases, where conditions~(3) and~(4) may fail
to hold, the essentially are always satisfied in the cases of interest (at least in dimension~$3$). This is due
to the basic properties of the standard reduction procedure by means of blow-up maps
for singularities in dimension~$3$ (which is known to the expert to hold in dimension~$3$ provided that natural topological
conditions are satisfied). However, to avoid making the discussion needlessly long, we
shall proceed as in \cite{scardua} and assume once and for all that, up to a rotation, there is an eigenvalue
$\lambda_1$ lying in $\R_+$ whereas the remaining eigenvalues $\lambda_2, \ldots ,\lambda_n$ lie in $\R_-$.
We can now state Theorem~\ref{whatever}.

\begin{theorem}
\label{whatever}
Let $\fol$ denote a singular foliation on $(\C^n,0)$ whose eigenvalues $\lambda_1, \ldots , \lambda_n$ are all
different from zero. Suppose also that $\lambda_1 \in \R_+$ while $\lambda_2, \ldots ,\lambda_n$ are all negative reals.
Denote by $h_1$ the local holonomy map associated to the axis $x_1$ (corresponding to the eigenvalue $\lambda_1$)
and suppose that $h_1$ has isolated fixed points (in the sense of Theorem~A) and that it has finite orbits.
Then $\fol$ admits $n-1$ independent holomorphic first integrals.
\end{theorem}

To prove Theorem~\ref{whatever}, Theorem~\ref{TMMhigher} below will be needed. The latter theorem generalizes
to higher dimensions a unpublished result of Mattei which, in turn, improved on an earlier version appearing in \cite{M-M}.

\begin{theorem}\label{TMMhigher}
{\rm ({\bf [EI],  [Re]})}
Let $X$ and $Y$ be two vector fields satisfying conditions (1), (2), (3) and~(4) above.
Denote by $h^X$ (resp. $h^Y$) the holonomy of $X$ (resp. $Y$) relatively to
the separatrix of $X$ (resp. $Y$) tangent to the eigenspace associated to the
first eigenvalue. Then $h^X$ and $h^Y$ are analytically conjugate if and only
if the foliations associated to $X$ and $Y$ are analytically equivalent.
\end{theorem}

The proof of Theorem~\ref{TMMhigher} can be found in either \cite{EI} or \cite{helena}, a particularly detailed exposition
appears in \cite{monograph}.
With this theorem in hand, the proof of Theorem~\ref{whatever} goes as follows.

\begin{proof}[Proof of Theorem~\ref{whatever}]
Consider a foliation $\fol$ as in the statement of Theorem~\ref{whatever}. The local holonomy map $h_1$ is defined
on a suitable local section and it can also be identified to a local diffeomorphism fixing the origin
of $\C^{n-1}$. By assumption, all iterates of $h$ have isolated fixed points. Therefore Theorem~A implies
that the local orbits of $h$ are finite if and only if $h$ is periodic. Naturally, we may
assume this to be the case. Let then $N$ be the {\it period}\, of $h$, namely the smallest strictly positive
integer for which $h^N$ coincides with the identity on a neighborhood of the origin of $\C^{n-1}$
(with the above mentioned identifications). Denote also by $T$ the derivative of $h$ at the origin, which
is itself identified to a linear transformation of $\C^{n-1}$. The fact that $h$ is periodic of
period~$N$ ensures that $T$ is also periodic with the same
period~$N$. In fact, $h$ and $T$ are analytically conjugate as already mentioned (i.e. $h$ is linearizable).
Next denote by $\fol_Z$ the foliation associated to the linear vector field
$$
Z = \sum_{i=1}^n \dl_ix_i \partial /\partial x_i \, .
$$
It is immediate to check that the map $T$ coincides with the holonomy map induced by $\fol_Z$ with respect to
the axis $x_1$. Therefore Theorem~\ref{TMMhigher} implies that the foliations $\fol$ and $\fol_Z$ are analytically
equivalent. However, since $\fol_Z$ is induced by a linear (diagonal) vector field, it becomes clear that the
complete integrability of $\fol_Z$ is equivalent to the periodic
character of the holonomy map $T$. Since $\fol$ and $\fol_Z$ are analytically equivalent, we conclude from what precedes that the condition
of having a local holonomy $h$ with finite orbits forces $\fol$ to be completely integrable. The converse is clear, since having $\fol$
completely integrable ensures at once that the holonomy map $h$ must be periodic. Theorem~\ref{whatever} is proved.
\end{proof}

\begin{obs}
{\rm Concerning the argument given in \cite{scardua} for the analogous statement in dimension~$3$, the authors
have relied in Abate's theorem to conclude that the corresponding holonomy map must be of finite order. Abate's theorem
however is no longer available once the dimension is~$4$ or larger.}
\end{obs}


\section{Examples of local dynamics and proof of Theorem~B}

This section is split in two paragraphes. First we shall describe the local dynamics of certain
local diffeomorphisms fixing the origin of $\C^2$ that happen to be tangent to the identity. These
examples contains family of diffeomorphisms having finite orbits whereas the diffeomorphism
itself has infinite order.
In the second paragraph we are going to prove Theorem~B by showing that, in fact, the holonomy
map associated to the axis $\{ x=y=0\}$ in the example described in Theorem~B falls in one of the
previously discussed classes of diffeomorphisms.

\subsection{Local diffeomorphisms}

Recall that $\diffCtwo$ denotes the normal subgroup of $\Diffgentwo$ consisting of diffeomorphisms tangent to the identity.
The simplest example of an element of $\diffCtwo$ possessing finite orbits is obtained by setting
$F(x,y) = (x + f(y) , y)$, with $f(0) =f'(0) =0$. This case, however, can be set aside in what follows since
the foliation associated to the infinitesimal generator of $F$ is regular. In particular,
it cannot be realized as the holonomy map of a Siegel singular point. Some genuinely more interesting examples
are listed below.

\vspace{0.1cm}

\noindent {\bf Example 1}: Linear vector fields.

Consider the vector field $Y$ given by $Y = x \partial /\partial x -  \lambda y \partial /\partial y$
where $\lambda =n/m$ with $m,n \in \N^{\ast}$. The foliation associated to $Y$ will be denoted by $\fol$ and it should be noted
that the holomorphic function $(x,y) \mapsto x^n y^m$ is a first integral for $\fol$. Let $\phi_Y$ denote the time-one map
induced by $Y$. The local dynamics of $\phi_Y$ can easily be described as follows. The vector field $Y$ can be projected on the axis
$\{ y=0\}$ as the vector field $x \partial /\partial x$. Therefore the (real) integral curves of $Y$ coincide
with the lifts in the corresponding leaves of $\fol$ of the (real) trajectories of $x \partial /\partial x$ on $\{ y=0\}$.
The latter trajectories are radial lines being emanated from $0 \in \{ y=0\} \simeq \C$ so that
the local dynamics of $\phi_Y$ restricted to $\{ y=0\}$ is such that, whenever $x_0 \neq 0$, the sequence $\{\phi_Y^n (x_0) \}$ marches off a
uniform neighborhood of $0 \in  \{ y=0\} \simeq \C$ as $n \rightarrow \infty$ and it converges to $0 \in  \{ y=0\} \simeq \C$
as $n \rightarrow -\infty$. Consider now the orbit of
a point $(x_0, y_0)$, $x_0y_0 \neq 0$, by $\phi_Y$. Since this is simply the lift in the leaf of $\fol$ through $(x_0, y_0)$ of the dynamics
of $x_0 \in \{ y=0\} \simeq \C$, it follows that $\phi_Y^n (x_0, y_0)$ leaves a fixed
neighborhood of $(0,0) \in \C^2$ since the first coordinate increases to uniformly large values provided that $n \rightarrow \infty$.
Similarly, when $n \rightarrow -\infty$, the first coordinate of $\phi_Y^n (x_0, y_0)$ must converge to
{\it zero}\, so that the second coordinate becomes ``large'' due to the first integral $x^ny^m$. Thus, fixed a (small) neighborhood
$U$ of $(0,0) \in \C^2$, every orbit of $\phi_Y$ that is not contained in $\{ x=0\} \cup \{ y=0\}$ is bound to intersect $U$ at finitely many points.

Clearly the time-one map induced by $Y$ is not tangent to the identity. However, examples of time-one maps
tangent to the identity and satisfying the desired conditions can be obtained, for example, by considering
the vector field $X = x^ny^m Y$ and taking the time-one map $\phi_X$ induced by $X$.
Clearly the linear part of $X$ at $(0,0)$ equals zero so that $\phi_X$ must be tangent to the identity. Furthermore,
the multiplicative factor $x^ny^m$ annihilates the dynamics of $\phi_X$ over the coordinate axes so that only
the orbits of points $(x_0, y_0)$ with $x_0y_0 \neq 0$ have to be considered.
The leaf of $\fol$ through $(x_0, y_0)$ will be denoted by $L$ and $c \in \C$ will stand for the value of $x^ny^m$ on $L$.
The restriction of $X$ to $L$ is nothing but the restriction of $Y$ to $L$ multiplied by the scalar
$c \in \C$. Therefore the real orbits of $X$ in $L$ coincide with the lift to $L$ of the real orbits of the vector field
$cx  \partial /\partial x$ defined on $\{ y=0\}$.  The geometric nature of the orbits of $cx \partial /\partial x$
depends on the argument of $c \in \C$, i.e. setting $c= \vert c \vert
e^{2\pi i \alpha}$, this geometry depends on $\alpha \in [0, 2\pi)$. If $\alpha = \pi/2$, then the orbits of
$cx \partial /\partial x$ are contained in circles about the origin.
After finitely many tours, these circles lift into the corresponding leaf (i.e. the leaf on which $xy$ equals~$c$)
as closed paths invariant by $\phi_X$. In addition for a ``generic'' choice of $c$ satisfying $\alpha = \pi/2$,
the resulting time-one map restricted to the corresponding invariant path will be conjugate to an irrational rotation. Thus
$\Phi_X$ does not have finite orbits.

Let us now briefly discuss the slightly more general case where $X =x^a y^b Y$ with $a,b \in \N^{\ast}$.
Setting $d = am-bn$, we can suppose without loss of generality that $d\geq 1$. Next, by considering the system
$$
\begin{cases}
\frac{dx}{dt} = mx^{a+1} y^b \\
\frac{dy}{dt} = -nx^a y^{b+1} \, ,
\end{cases}
$$
we conclude that $dy/dx = -ny/mx$ so that $y = c x^{-n/m}$ in ramified coordinates. In turn, this
yields $dx/dt = c^b mx^{1 + d/m} \partial /\partial x$. Since $d \geq 1$ by construction, the orbits of the latter
vector field defines the well-known ``petals'' associated to Leau flower in the case of periodic linear
part, cf. \cite{carlerson}. For example, setting $m=1$ to simplify,
the orbits of the vector field $x^{1+d} \partial /\partial x$ consists of $d+1$ ``petals'' in non-ramified coordinates.

In any event, the sequence of points consisting of the first coordinates of the
full orbits of $\phi_X$ either marches straight off a neighborhood of $0 \in \{ y=0\} \simeq \C$ or it converges to $(0,0)$
(by construction this sequence is contained in $\{ y=0\}$).
Converging to $(0,0)$ will force the second coordinates of the points in the $\phi_X$-orbit to increase uniformly so that
the orbit in question must leave a fixed neighborhood of $(0,0) \in \C^2$. Summarizing, we conclude:

\noindent {\it Claim 1}. Fixed a neighborhood $U$ of $(0,0) \in \C^2$ and given $p = (x_0, y_0)$, $x_0y_0 \neq 0$, the set
$$
U \cap \left\{ \bigcup_{n=-\infty}^{\infty} \phi_X^n (p) \right\}
$$
is finite.

\vspace{0.1cm}

\noindent {\bf Example 2}: Diffeomorphisms leaving the function $(x,y) \mapsto xy$ invariant.

Here we are going to see two examples of diffeomorphisms having a nature somehow similar to those discussed in Example~1
but having also the advantage that they can easily be realized as the holonomy maps
of foliations with Siegel singular points. To begin with, let $F \in \diffCtwo$ be given by
\begin{equation}
F (x,y) = [ x(1+xy f(xy)), \, y (1 +xy f(xy))^{-1}] \, , \label{forholonomy1}
\end{equation}
where $f (z)$ is a holomorphic function defined about $0 \in \C$ and satisfying $f(0) \neq 0$.
Note that $F$ leaves the function $(x,y) \mapsto xy$ invariant since the product of its first and second components
equals~$xy$.

Next, consider an initial point $(x_0,y_0)$ with $x_0y_0 =C \neq 0$. The orbit of $(x_0,y_0)$
under $F$ is hence contained in the curve defined by $\{ xy = C\}$.
Moreover, for a point $(\tilde{x} , \tilde{y})$ lying in $\{ xy = C\}$, the value of
$F (\tilde{x} , \tilde{y})$ takes on the form
$$
F (\tilde{x} , \tilde{y}) = [\tilde{x} (1 + C  f(C)) \, , \,  \tilde{y} (1 + C  f(C))^{-1} ] \, .
$$
In particular, those values of $C$ for which $\vert 1 + C  f(C) \vert =1$ give rise to a rotation in the first coordinate.
Therefore the lifts of these circles in the corresponding leaves are loops. Besides for a generic choice of $C$
satisfying $\vert 1 + C  f(C) \vert =1$ the dynamics
induced on one of these invariant loops is conjugate to an irrational rotation so that $F$ does not have finite orbits.

Consider now the local diffeomorphism $H$ which is given by
\begin{equation}
H (x,y) = [ x(1+x^2y f(x^2y)), \, y (1 +x^2y f(x^2y))^{-1}] \, , \label{forholonomy2}
\end{equation}
where $f$ is as above. For this local diffeomorphism, we have:

\begin{lemma}
\label{diffeoforthmB}
The local diffeomorphism $H$ given by Formula~(\ref{forholonomy2}) possesses only finite orbits.
\end{lemma}

\begin{proof}
Again the product of its first and second components of $H$ equals~$xy$ so that $H$ preserves the function
$(x,y) \mapsto xy$. To check that $H$ has finite orbits, we proceed as follows.
Fix an initial point
$(x_0,y_0)$ with $x_0y_0 =C \neq 0$ so that the orbit of $(x_0,y_0)$
under $H$ is contained in the curve $\{ xy = C\}$. Next note that, if $(\tilde{x} , \tilde{y})$ lies in $\{ xy = C\}$, we have
$$
H (\tilde{x} , \tilde{y}) = [\tilde{x}(1 + \tilde{x} C f( \tilde{x} C) \, , \,  \tilde{y} (1 + \tilde{x}C f(\tilde{x} C))^{-1} ] \, .
$$
The dynamics of the first component of $H$ behaves now as the Leau flower since it is given by $x \mapsto x + x^2 C f(xC)$,
where $C \neq 0$. Therefore, by resorting to an argument totally
analogous to the one employed in Example~1 for $X =x^a y^b Y$ with $d = am-bn \neq 0$, we conclude that all the orbits
of $H$ are finite as desired.
\end{proof}


\subsection{Singular foliations and holonomy}

In this paragraph Theorem~B will be proved and a couple of related examples will also be provided.

Let us begin by pointing out a simple observation showing that every element of $\diffCtwo$ can be realized as a local
holonomy map for some foliation on $(\C^3,0)$. Indeed, consider
a singular foliation $\fol$ on $(\C^3,0)$ admitting a separatrix $S$ through the origin and denote by
$h$ the holonomy map associated to $\fol$, with respect to $S$. Assume that the foliation is locally given by
the vector field $A(x,y,z) \partial /\partial x + B(x,y,z) \partial /\partial y + C(x,y,z) \partial /\partial
z$. Assume furthermore that the separatrix $S$ is given, in the same coordinates, by $\{x=0, y=0\}$. Setting $z = e^{2\pi it}$, the corresponding
holonomy map can be viewed as the time-one map associated to the differential equation
\[
\begin{cases}
\frac{dx}{dt} & = \frac{dx}{dz} \frac{dz}{dt} = 2\pi i e^{2\pi i t} \frac{A(x,y,e^{2 \pi i t})}{C(x,y,e^{2\pi it})}\\
\frac{dy}{dt} & = \frac{dy}{dz} \frac{dz}{dt} = 2\pi i e^{2\pi i t} \frac{B(x,y,e^{2 \pi i t})}{C(x,y,e^{2\pi it})}
\end{cases} \, .
\]
In the particular case where $A, \, B$ do not depend on $z$ and $C$ is reduced to $C(x,y,z) = z$, the holonomy
map of $\fol$ with respect to $S$ reduces to the time-one map induced by a vector field on $(\C^2, 0)$, namely by
the vector field
\[
2\pi i \left[ A(x,y) \frac{\partial}{\partial x} + B(x,y) \frac{\partial}{\partial y} \right] \, .
\]

Considering a local diffeomorphism $h$ possessing finite orbits and realizable as time-one map of a vector field $Y$,
then to find a vector field on $(\C^3,0)$ whose
foliation has $h$ as holonomy map, it suffices to take $Y$ and ``join" the term $2\pi i z \partial
/\partial z$. Then the holonomy of the foliation associated to the vector field
\[
X = Y + 2\pi i z \frac{\partial }{\partial z} \, ,
\]
with respect to the $z$-axis is nothing but $h$ itself.

Note that the vector field $X$ above corresponds to a saddle-node vector field of codimension~$2$. This is
equivalent to saying that its linear part admits exactly two eigenvalues equal to zero and a non-zero
eigenvalue associated to the direction of the separatrix $\{ x=y=0\}$. The fact that the holonomy of $\{ x=y=0\}$
has finite orbits is a phenomenon without analogue for saddle-nodes in dimension~$2$.

Let us now consider three special examples of foliations on $(\C^3,0)$ possessing only eigenvalues different from
{\it zero}, as in the case of Theorem~B.

\noindent {\bf Example 3}: Let $\fol$ denote the foliation associated to the vector field
\[
X = x(1 + xyz^2) \frac{\partial }{\partial x} + y(1 - xyz^2) \frac{\partial }{\partial y} - z \frac{\partial }{\partial z} \, .
\]
The $z$-axis corresponds to one of the separatrices of $\fol$. Taking $z = e^{2\pi it}$, it follows that the holonomy map $h$
associated to $\fol$, with respect to the $z$-axis, is given by the time-one map associated to the vector field
\begin{equation}\label{eqholonomy}
\begin{cases}
\frac{dx}{dt} & = \frac{dx}{dz} \frac{dz}{dt} = -2\pi i x(1 + e^{4\pi it}xy)\\
\frac{dy}{dt} & = \frac{dy}{dz} \frac{dz}{dt} = -2\pi i y(1 - e^{4\pi it}xy)
\end{cases} \, .
\end{equation}
To solve this system of differential equations, we should consider the series expansion of $(x(t),y(t))$ in terms of the
initial condition. More precisely, if $(x(0), y(0)) = (x_0, y_0)$, then we should let $x(t) = \sum a_{ij}(t) x_0^i y_0^j$
and $y(t) = \sum b_{ij}(t) x_0^i y_0^j$. Clearly $a_{10}(0) = b_{01}(0) = 1$ and $a_{ij} = b_{ij}(0) = 0$ in the other
cases. Substituting the series expansion of $x(t)$ and $y(t)$ on~(\ref{eqholonomy}) and comparing the same powers on the
initial conditions, it can be said that the system~(\ref{eqholonomy}) induces an infinite number of differential equations
involving the functions $a_{ij}, \, b_{ij}$ and their derivatives. Each one of the differential equation takes on the form
\begin{eqnarray*}
a_{ij}^{\prime}(t) & = & -2\pi i  \left[ a_{ij}(t) + \sum e^{4\pi it} a_{p_1q_1}(t) a_{p_2q_2} (t) b_{p_3q_3}(t) \right] \\
b_{ij}^{\prime}(t) & = & -2\pi i  \left[ b_{ij}(t) + \sum e^{4\pi it} a_{p_1q_1}(t) b_{p_2q_2} (t) b_{p_3q_3}(t) \right]
\end{eqnarray*}
where $p_1 + p_2 +p_3 =i$ and $q_1 + q_2 + q_3 = j$. In particular, the terms on the sum in the right hand side of the equation
above involves only coefficients of the monomials $x_0^p y_0^q$ of degree less then $i + j$ and such that $p \leq i$
and $q \leq j$. Computing this holonomy map becomes much easier with the following lemma:
\begin{lemma}
\label{preservingxy}
The holonomy map $h$ preserves the function $(x,y) \mapsto xy$.
\end{lemma}

\begin{proof}
To check that the level sets of $(x,y) \mapsto xy$ are preserved by $h$, consider the derivative of the product $x(t) y(t)$ with respect
to $t$. This gives us
$$
\frac{d}{dt} (xy) = \frac{dx}{dt}  y + x  \frac{dy}{dt}
= - \left[ 2\pi i x(1 + e^{4\pi it}xy) \right] y - x \left[ 2\pi i y(1 - e^{4\pi it}xy) \right]
= -4\pi i xy \, .
$$
Thus, by integrating the previous differential equation with respect to the product $xy$, we obtain
$(xy)(t) = x_0y_0 e^{-4 \pi i t}$.
Since the holonomy map corresponds to the time-one map of the system of differential equations~(\ref{eqholonomy}) and
since $e^{-4 \pi i t} = 1$ for all $t \in \Z$, it follows that the orbits of $h$ are contained in the level sets of $(x,y) \mapsto xy$ as desired.
\end{proof}

Lemma~\ref{preservingxy} implies that it suffices to determine the first coordinate of $h$.
By recovering the preceding non-autonomous system of differential equations, a simple induction argument
on the value of $i+j$ shows that $h$ has the form
\begin{equation}\label{eqexpressionH}
h(x,y) = (x(1 + xy f(xy)), y(1 + xy f(xy))^{-1}) \, ,
\end{equation}
where $f$ represents a holomorphic function of one complex variable such that $f(0) = 2\pi i$
(the expression for the second coordinate of $h$ is obtained from the first coordinate by means of
Lemma~\ref{preservingxy}). The resulting diffeomorphism $h$ is clearly non-periodic
but it does have invariant sets given by ``circles''.
Besides on some of these invariant ``circles'' the dynamics is conjugate to an irrational
rotation, cf. Example~2.

\vspace{0.1cm}

We are now ready to prove Theorem~B

\begin{proof}[Proof of Theorem B]
Let then $\fol$ denote the foliation associated to the vector field
\[
X = x(1 + x^2yz^3) \frac{\partial }{\partial x} + y(1 - x^2yz^3) \frac{\partial }{\partial y} - z \frac{\partial }{\partial z} \, .
\]
Again the $z$-axis corresponds to one of the separatrices of $\fol$.
Taking $z = e^{2\pi it}$, it follows that the holonomy map $h$
associated to $\fol$, with respect to the $z$-axis, is given by the time-one map associated to the vector field
\begin{equation}\label{eqholonomy2}
\begin{cases}
\frac{dx}{dt} & = \frac{dx}{dz} \frac{dz}{dt} = -2\pi i x(1 + e^{6\pi it}x^2y)\\
\frac{dy}{dt} & = \frac{dy}{dz} \frac{dz}{dt} = -2\pi i y(1 - e^{6\pi it}x^2y)
\end{cases} \, .
\end{equation}
The same argument employed in Lemma~\ref{preservingxy} shows that the holonomy map $h$ in question preserves
the level sets of the function $(x,y) \mapsto xy$.
To solve the corresponding system of equations, we consider again the series expansion of $(x(t) ,y(t))$
in terms of the initial condition. Let then $x(t) = \sum a_{ij} (t) x_0^i y_0^j$ and
$y(t) = \sum b_{ij} (t) x_0^i y_0^j$, where $a_{10} (0) = b_{01} (0) =1$ and $a_{ij} (0) =b_{ij} (0) =0$
in the remaining cases. It can immediately be checked that the functions
$a_{ij}, b_{ij}$ vanish identically for $2 \leq i+j \leq 3$. As to the monomials of degree~$4$, it can similar
be checked that they all vanish identically except $a_{31} (t)$ and $b_{22} (t)$. In fact, the latter functions satisfy
\begin{equation}
\begin{cases}
a_{31}^{\prime} (t) & =  -2\pi i [ a_{31} (t) +  e^{6\pi i} a_{10}^3 (t) b_{01} (t)] \\
b_{22}^{\prime} (t) & = -2\pi i [ b_{22} (t) - e^{6\pi i} a_{10}^2 (t) b_{01}^2 (t)]
\end{cases}
\end{equation}
so that $a_{31} (t) = -2\pi i t e^{-2\pi it}$ whereas $b_{22} (t) = 2\pi i t e^{-2\pi it}$.
In particular, $a_{31} (1) = -2\pi i$ and $b_{22} (1) = 2\pi i$. By using induction on $i+j$ (and keeping in mind
that $h$ preserves the function $(x,y) \mapsto xy$), it can be shown that $h$ takes on the form
\begin{equation}\label{eqexpressionH2}
h(x,y) = (x(1 + x^2y f(x^2y)), y(1 + x^2y f(x^2y))^{-1}) \; \; \, {\rm with} \; \; \, f(0) =2\pi i \, .
\end{equation}
It follows from Lemma~\ref{diffeoforthmB} that this diffeomorphism possesses finite orbits whereas it is clearly non-periodic.
This finishes the proof of Theorem~B.
\end{proof}

\vspace{0.1cm}


\bigskip

\bigskip

\begin{flushleft}
{\sc Julio Rebelo} \\
Institut de Math\'ematiques de Toulouse\\
118 Route de Narbonne\\
F-31062 Toulouse, FRANCE.\\
rebelo@math.univ-toulouse.fr

\end{flushleft}

\bigskip

\begin{flushleft}
{\sc Helena Reis} \\
Centro de Matem\'atica da Universidade do Porto, \\
Faculdade de Economia da Universidade do Porto, \\
Portugal\\
hreis@fep.up.pt \\

\end{flushleft}

\end{document}